\documentclass[11pt,reqno]{amsart}

\usepackage{amssymb,amsrefs,amsmath,mathtools}
\usepackage{a4wide,graphicx,nicefrac}
\usepackage[shortlabels]{enumitem}
\usepackage{mathrsfs}
\usepackage{hyperref}
\usepackage{ourbib}
\usepackage{xcolor}

\setlist{leftmargin=4mm}
\newtagform{simple}{}{}{}

\hypersetup{
    colorlinks,
    linkcolor={red!50!black},
    citecolor={blue!50!black},
    urlcolor={blue!80!black}
}

\newcommand{\HH}{\mathscr{H}}
\newcommand{\RadSn}{\mathsf{Rad}_{\Sn}}
\newcommand{\RadX}{\mathsf{Rad}_{\XX}}
\newcommand{\diam}{\mathsf{diam}}
\newcommand{\Lip}{\mathsf{Lip}}
\newcommand{\ns}{\nabla^s}
\newcommand{\nn}{\nabla^0}
\renewcommand{\sf}{\mathrm{sf}}
\newcommand{\ZZ}{\mathscr{Z}}
\newcommand{\spindim}{2^{\lfloor n/2\rfloor}}
\newcommand{\Ctriv}{\underline{\mathbb{C}}^{\spindim}}
\newcommand{\Sn}{\mathbb{S}^n}
\newcommand{\XX}{\mathfrak{X}}

\newcommand{\inj}{\mathrm{inj}}
\newcommand{\tr}{\mathrm{tr}} 
\newcommand{\R}{\mathbb{R}}

\newcommand{\C}{\mathbb{C}}
\newcommand{\dX}{{\partial X}}
\newcommand{\D}{\mathsf{D}}
\newcommand{\Y}{\mathsf{Y}}
\newcommand{\eps}{\varepsilon}
\renewcommand{\phi}{\varphi}
\newcommand{\Ahat}{\mathsf{\hat A}}
\newcommand{\ch}{\mathsf{ch}}
\newcommand{\e}{\mathsf{e}}
\newcommand{\scal}{\mathrm{scal}}
\newcommand{\<}{\langle}
\renewcommand{\>}{\rangle}
\newcommand{\vol}{\mathsf{vol}}
\newcommand{\myicon}{$\,\,\,\triangleright$}

\newcommand{\listdisplay}[1]{\vspace{2mm}\hfill \( \displaystyle #1\)\hfill\mbox{}\vspace{1mm}}

\DeclareMathOperator{\ind}{\mathrm{ind}}
\DeclareMathOperator{\id}{\mathrm{id}}
\DeclareMathOperator{\dist}{\mathrm{dist}}

\newtheorem{theorem}{Theorem}
\newtheorem*{maintheorem}{Main Theorem}
\newtheorem{lemma}{Lemma}

\theoremstyle{definition}
\newtheorem{remark}{Remark}
\newtheorem{definition}{Definition} 
\newtheorem{example}{Example} 
\newtheorem{question}{Question} 
\allowdisplaybreaks

\begin{document}

\title{Dirac eigenvalues and the hyperspherical radius} 
\author{Christian B\"ar}
\address{Universit\"at Potsdam, Institut f\"ur Mathematik, 14476 Potsdam, Germany}
\email{\href{mailto:christian.baer@uni-potsdam.de}{christian.baer@uni-potsdam.de}}
\urladdr{\url{https://www.math.uni-potsdam.de/baer/}}

\begin{abstract} 
For closed connected Riemannian spin manifolds an upper estimate of the smallest eigenvalue of the Dirac operator in terms of the hyperspherical radius is proved.
When combined with known lower Dirac eigenvalue estimates, this has a number of geometric consequences.
Some are known and include Llarull's scalar curvature rigidity of the standard metric on the sphere, Geroch's conjecture on the impossibility of positive scalar curvature on tori and a mean curvature estimate for spin fill-ins with nonnegative scalar curvature due to Gromov, including its rigidity statement recently proved by Cecchini, Hirsch and Zeidler.
New applications provide a comparison of the hyperspherical radius with the Yamabe constant and improved estimates of the hyperspherical radius for Kähler manifolds, Kähler-Einstein manifolds, quaternionic Kähler manifolds and manifolds with a harmonic $1$-form of constant length.
\end{abstract}



\date{\today}

\maketitle

\section{Introduction} 

The idea of hyperspherical radius of a closed connected orientable manifold $M$ goes back to Gromov and Lawson in \cite{GL} where it was used to prove Geroch's conjecture on the impossibility of positive scalar curvature on tori.
It is defined as the largest number $R>0$ such that $M$ still dominates the round sphere $\Sn(R)$ of radius $R$ in the sense that there exists a $1$-Lipschitz map $M\to\Sn(R)$ with nontrivial degree.
Here $n=\dim(M)$.

If, in addition, $M$ carries a spin structure, then the Dirac operator $\D$ is defined on $M$.
It is elliptic and selfadjoint and hence has discrete real spectrum.
We denote the smallest eigenvalue of $\D^2$ by $\lambda_1(\D^2)$.
The main result of the present paper is a bound for $\lambda_1(\D^2)$ in terms of the hyperspherical radius.

\begin{maintheorem}
Let $M$ be a connected closed Riemannian spin manifold of dimension $n\ge2$.
Denote the Dirac operator acting on spinor fields of $M$ by $\D$.
Then
\begin{equation}
\lambda_1(\D^2) \le \frac{n^2}{4\,\RadSn(M)^2}.
\label{eq.DiracRad}
\end{equation}
Equality holds in \eqref{eq.DiracRad} if and only if $M$ is isometric to $\Sn(R)$ with $R=\RadSn(M)$.
\end{maintheorem}

In fact, every $1$-Lipschitz map $f\colon M\to \Sn(R)$ with $R=\RadSn(M)$ and $\deg(f)\neq0$ turns out to be a smooth Riemannian isometry.

In general, the hyperspherical radius is hard to control.
This can even be made precise in complexity theoretical terms, see \cite{BGM}.
The main theorem has a number of applications which are obtained by combining it with known lower bounds for $\lambda_1(\D^2)$, which results in upper bounds for the hyperspherical radius.
Amongst these applications are some classically known theorems such as Llarull's scalar curvature rigidity of the standard metric of $\Sn$, the already mentioned Geroch's conjecture and an estimate for the mean curvature of spin fill-ins with nonnegative scalar curvature due to Gromov \cite{G4}, including its rigidity statement recently proved by Cecchini, Hirsch and Zeidler \cite{CHZ}.
New applications are concerned with an estimate between the hyperspherical radius and the Yamabe constant as well as improved estimates for the hyperspherical radius on manifolds the special structures, more precisely, on Kähler manifolds, Kähler-Einstein manifolds, quaternionic Kähler manifolds and manifolds admitting a nontrivial harmonic $1$-form of constant length. 

The paper is organized as follows.
In Section~\ref{sec.HypersphericalRadius} we recall the definition of the hyperspherical radius and discuss a number of its properties which follow directly from the definition.
In Section~\ref{sec.applications} we discuss applications of the main theorem.
The next two sections are devoted to the proof of the main theorem.
In Section~\ref{sec.LipschitzTwists} we provide some necessary technical tools and in Section~\ref{sec.spherrad} we carry out the proof.
One is not necessarily restricted to compare the manifold $M$ with round spheres.
In Section~\ref{sec.X} we describe a class of admissible comparison spaces in even dimensions for which an analogue to the main theorem can still be proved (Theorem~\ref{thm.DiracRadX}).
This class includes the sphere with a metric originating from boundaries of convex domains in Euclidean space.
The odd-dimensional case is left open in this more general setup.
We conclude the paper with a number of remarks and open questions in Section~\ref{sec.conclusion} and an appendix which contains a complete proof of \cite{HMZ}*{Theorem~6} for the convenience of the reader.

\medskip

\textit{Acknowledgments:} 
The author would like to thank Bernhard Hanke, Lennart Ronge and Rudolf Zeidler for many insightful discussions and ``SPP2026 - Geometry at Infinity'' funded by Deutsche Forschungsgemeinschaft for financial support.

\section{The hyperspherical radius}
\label{sec.HypersphericalRadius}

\begin{definition}\label{def.hyprad}
Let $M$ be a closed connected $n$-dimensional Riemannian manifold.
The \emph{hyperspherical radius} $\RadSn(M)$ of $M$ is the supremum of all numbers $R>0$ such that there exists a Lipschitz map $f\colon M\to S^n$ with Lipschitz constant $\Lip(f)\le 1/R$ and $\deg(f)\ne0$.
\end{definition}

Equivalently, $\RadSn(M)$ is the supremum of all numbers $R>0$ such that there exists a $1$-Lipschitz map $f\colon M\to \Sn(R)$ with nonzero degree.
Here $\Sn(R)$ denotes the round sphere of radius $R$.

\begin{remark}
The supremum in Definition~\ref{def.hyprad} is a maximum.
In other words, there exists a Lipschitz map $f\colon M\to \Sn$ with $\Lip(f)= 1/\RadSn(M)$ and $\deg(f)\neq0$.

Indeed, let $R_i\nearrow \RadSn(M)$ and let $f_i\colon M\to \Sn$ be Lipschitz maps with $\Lip(f_i)\le 1/R_i$ and $\deg(f_i)\neq0$.
By the Arzelà-Ascoli theorem, $f_i$ subconverges uniformly to a Lipschitz map $f$ with $\Lip(f)\le 1/\RadSn(M)$.
Since the (absolute value of) the degree of $f_i$ is at most $\Lip(f_i)^n$, the degrees of the $f_i$ are bounded.
Thus, after passing to a subsequence once more, all $f_i$ have the same nontrivial degree.
The limit map then also has this degree $\deg(f)\ne0$.
Since there can be no Lipschitz map $g\colon M\to \Sn$ with $\Lip(g)<1/\RadSn(M)$ and $\deg(g)\neq0$, we conclude:
there exists a Lipschitz map $f\colon M\to \Sn$ with $\Lip(f)=1/\RadSn(M)$ and $\deg(f)\neq0$.
\end{remark}

\begin{remark}
If $f\colon M\to\Sn$ is a Lipschitz map with $\Lip(f)<1/R$ for some $R<\RadSn(M)$, then we can approximate $f$ by a smooth map $\tilde{f}\colon M\to \Sn$ such that $\Lip(\tilde{f})<1/R$ and $\deg(\tilde{f})=\deg(f)$.
Therefore, we can equivalently define the hyperspherical radius as the supremum of all $R>0$ for which there exists a \emph{smooth} map $f\colon M\to S^n$ with $\Lip(f)\le 1/R$ and $\deg(f)\ne0$.
Then, however, the supremum may no longer be a maximum.
\end{remark}

\begin{remark}\label{rem.raddiam}
If $f\colon M\to S^n$ is a Lipschitz map with $\Lip(f)\le 1/R$, then the image of a ball $B(p,r)\subset M$ under $f$ must lie in the ball $B(f(p),\frac{r}{R})$.
Thus, if $\frac{\diam(M)}{R}< \diam(S^n) = \pi$, then $f$ is not surjective and hence has degree $0$.
This shows 
$$
\RadSn(M)\le \frac{\diam(M)}{\pi}.
$$
\end{remark}

\begin{remark}
Let $f\colon M\to\Sn$ be Lipschitz with $\Lip(f)=1/R$ and $\deg(f)\neq0$ where $R=\RadSn(M)$.
Without loss of generality assume $\deg(f)>0$.
Denote the volume forms of $M$ and $\Sn$ by $\vol_{M}$ and $\vol_{\Sn}$, respectively.
At points where $f$ is differentiable we have $f^*\vol_{\Sn} = \pm \mu_1\cdots \mu_n\cdot \vol_{M}$ where $\mu_1,\dots,\mu_n$ are the singular values of $df$ and the sign is positive if and only if $df$ preserves the orientation.
From $\mu_j\le 1/R$ we find
\begin{align}
\vol(M) 
&=
\int_M \vol_M \ge
R^n \int_M \mu_1\cdots\mu_n \vol_M \label{eq.VolEst1}\\
&\ge 
R^n \int_M f^*\vol_{\Sn} 
= 
R^n \deg(f) \int_{\Sn} \vol_{\Sn} \label{eq.VolEst2}\\
&\ge 
R^n\, \vol(\Sn).\label{eq.VolEst3}
\end{align}
This implies
\begin{equation}
\RadSn(M)^n \le \frac{\vol(M)}{\vol(\Sn)}.
\label{eq.VolumeEstimate}
\end{equation}
If we have equality in \eqref{eq.VolumeEstimate}, then we have equality in \eqref{eq.VolEst1}--\eqref{eq.VolEst3}.
Equality in \eqref{eq.VolEst1} implies that $\mu_1=\dots=\mu_n=R$ almost everywhere, i.e.\ $df$ is $R$ times a linear isometry at almost all points of $M$.
Equality in \eqref{eq.VolEst2} implies that $df$ preserves the orientation almost everywhere. 
Theorem~1 in \cite{KMS} now implies that $f\colon M\to S^n$ is a smooth Riemannian covering (up to rescaling).
By equality in \eqref{eq.VolEst3}, $f$ has degree $1$.
Hence, $f$ is an isometry, after rescaling the metric of $M$.
We have seen:

\emph{Equality holds in \eqref{eq.VolumeEstimate} if and only if $M$ is isometric to $\Sn(R)$.}

\end{remark}

\begin{remark}\label{rem.rescaleRad}
The hyperspherical radius behaves like a radius under rescaling of the metric.
Namely, for any constant $\lambda>0$ we have
$$
\RadSn(M,\lambda^2g) = \lambda\cdot\RadSn(M,g).
$$
\end{remark}

\begin{example}
Since $f=\id\colon \Sn\to \Sn$ has degree $1$ and Lipschitz constant $\Lip(\id)=1$, we have $\RadSn(\Sn)\ge1$.
By Remark~\ref{rem.raddiam}, we also have the inverse inequality.
Thus $\RadSn(\Sn)=1$.
Remark~\ref{rem.rescaleRad} now implies
$$
\RadSn(\Sn(R)) = R.
$$
\end{example}

\begin{remark}
Let $M$ be a connected closed orientable manifold and let $g_1$ and $g_0$ be two Riemannian metrics on $M$.
Then, if $g_1\ge g_0$ in the sense of bilinear forms, we have for any Lipschitz map $\Lip(f\colon (M,g_1)\to \Sn)\le\Lip(f\colon (M,g_0)\to \Sn)$.
Hence 
$$
\RadSn(M,g_1)\ge\RadSn(M,g_0).
$$
\end{remark}

\begin{remark}\label{rem.injrad}
Let $(M,g)$ be a connected closed oriented Riemannian, let $p\in M$ and let $0<\rho\le\inj(p)$ where $\inj(p)$ is the injectivity radius of $p$ in $M$.
We assume that the sectional curvature of $M$ satisfies $\sec\le\frac{\rho^2}{\pi^2}$ on the open ball $B(p,\rho)$.
We rescale the metric $g$ to $\tilde{g}=(\pi/\rho)^2g$.
With respect to this metric, we have $\tilde{B}(p,\pi)=B(p,\rho)$ and the sectional curvature satisfies $\widetilde{\sec}\le 1$ on $\tilde{B}(p,\pi)$.

Pick any point $q\in\Sn$ and a linear isometry $A\colon T_pM\to T_q\Sn$ (w.r.t.\ the metric $\tilde{g}$).
We define a map $f\colon (M,\tilde{g})\to \Sn$ by
\begin{equation*}
f(x):=
\begin{cases}
\exp_q^{\Sn}\circ A\circ (\widetilde{\exp}_p^{M})^{-1}(x), & \text{ if }x\in \tilde{B}(p,\pi),\\
-q, & \text{ else}.
\end{cases}
\end{equation*}
Jacobi field estimates together with the sectional curvature bound imply that $f$ is $1$-Lipschitz.
Therefore, $\RadSn(M,\tilde{g})\ge1$.
Remark~\ref{rem.rescaleRad} implies 
\begin{equation}
\RadSn(M,g) \ge \frac{\rho}{\pi} .
\label{eq.RadSninjrad}
\end{equation}
In particular, if we find a ball in $M$ on which the Riemannian exponential map is invertible and the sectional curvature satisfies $\sec\le0$, then the radius of this ball provides a lower bound for the hyperspherical radius of $M$.
\end{remark}

% \begin{remark}
% The supremum in Definition~\ref{def.hyprad} is a maximum.
% In other words, there exists a Lipschitz map $f\colon M\to S^n$ with $\Lip(f)\le 1/\RadSn(M)$ and $\deg(f)\neq0$.

% Indeed, let $R_i\nearrow \RadSn(M)$ and let $f_i\colon M\to S^n$ be Lipschitz maps with $\Lip(f_i)\le 1/R_i$ and $\deg(f_i)\neq0$.
% By the Arzelà-Ascoli theorem, after passing to a subsequence, $f_i$ converges uniformly to a Lipschitz map $f$ with $\Lip(f)\le 1/\RadSn(M)$.
% Since the (absolute value of) the degree of $f_i$ is at most $\Lip(f_i)^n$, the degrees of the $f_i$ are bounded.
% Thus, after passing to a subsequence once more, all $f_i$ have the same nontrivial degree.
% The limit map then also has this degree $\deg(f)\ne0$.
% Since there can be no Lipschitz map $g\colon M\to S^n$ with $\Lip(g)<1/\RadSn(M)$ and $\deg(g)\neq0$, we conclude:
% there exists a Lipschitz map $f\colon M\to S^n$ with $\Lip(f)=1/\RadSn(M)$ and $\deg(f)\neq0$.
% \end{remark}

\section{Applications of the main theorem}
\label{sec.applications} 

Combining the main theorem with known lower bounds for the first Dirac eigenvalue leads to a number of interesting consequences some of which are known theorems.

\subsection{Llarull's theorem}
\label{subsec.appl.Llarull}
Friedrich proved in \cite{F}*{Thm.~A} that if a closed Riemannian spin manifold has scalar curvature $\scal\ge s_0>0$ then 
\begin{equation}
\lambda_1(\D^2) \ge \frac14\frac{n}{n-1}s_0 .
\label{eq.Friedrich}
\end{equation}
Combining this with the main theorem yields
\begin{equation}
\RadSn(M)^2 \le \frac{n(n-1)}{s_0}.
\label{eq.Llarull}
\end{equation}
Thus if $s_0\ge n(n-1)$, then $\RadSn(M)\le1$.
In other words, if the scalar curvature of $M$ is at least that of the standard sphere, then all Lipschitz maps $f\colon M\to S^n(1)$ of nontrivial degree must satisfy $\Lip(f)\ge1$.
Moreover, if $\Lip(f)=1$, we must have equality in the main theorem and hence $f\colon M\to S^n(1)$ is an isometry.
This proves:

\begin{theorem}[Llarull]
Let $M$ be a closed connected Riemannian spin manifold of dimension $n\ge2$.
Suppose that the scalar curvature of $M$ satisfies $\scal\ge n(n-1)$.

Then each $1$-Lipschitz map $f\colon M\to \Sn$ with nonzero degree is a Riemannian isometry.
\hfill\qed
\end{theorem}

This theorem goes back to Llarull (\cite{Ll}*{Thm.~B}) where the map $f\colon M\to\Sn$ is assumed to be smooth.
We recover the version where $f$ may, a priori, be only Lipschitz.
This situation has been treated by Lee and Tam in \cite{LT}.
In that work, Ricci flow is employed to reduce to the smooth situation. 
A different approach is due to Cecchini, Hanke and Schick (\cite{CHS}) which so far seems to give only the even-dimensional case, but one can replace the $1$-Lipschitz condition by the weaker assumption that $df$ is nonincreasing on $2$-vectors.
This corresponds to Theorem~C in \cite{Ll}.
Both methods allow to also treat metrics of lower regularity on $M$.

\subsection{Geroch's conjecture}
We follow the conventions of \cite{GL} and call a connected closed differentiable manifold $M$ \emph{enlargeable} if for some Riemannian metric $g$ on $M$ and for every $R>0$ there exists a finite covering $\pi\colon \hat{M}\to M$ where $\hat{M}$ is spin and $\RadSn(\hat{M},\pi^*g)\ge R$.
Since any two Riemannian metrics on $M$ can be estimated by each other, we could equivalently demand the existence of the coverings for \emph{all} Riemannian metrics $g$ on $M$.

\begin{theorem}[Gromov-Lawson \cite{GL}*{Thm.~A}]
An enlargeable closed manifold does not admit a Riemannian metric of positive scalar curvature.
\end{theorem}

\begin{proof}
Assume that $M$ is enlargeable and has a Riemannian metric $g$ with positive scalar curvature.
By compactness, there exists a positive constant $s_0$ such that $\scal_{g}\ge s_0$.
Then for any spin covering $\pi\colon \hat{M}\to M$ we also have $\scal_{\pi^*g}\ge s_0$.
Estimate \eqref{eq.Llarull} yields $\RadSn(\hat{M},\pi^*g)\le \frac{n(n-1)}{s_0}$, contradicting enlargeability.
\end{proof}

The simplest examples of enlargeable manifolds are tori.
To see this, equip the torus with a flat metric, pass to coverings by larger tori and employ \eqref{eq.RadSninjrad}.
The nonexistence of positive scalar curvature metrics on tori answers a question attributed to Geroch.

\subsection{Mean curvature of spin fill-ins}\label{subsec.fillin}
Let $M$ be an $n$-dimensional closed Riemannian spin manifold.
Let $h\colon M\to\R$ be a smooth function.

\begin{definition}
A \emph{spin fill-in} of $(M,h)$ is a pair $(X,\Phi)$ where $X$ an $(n+1)$-dimensional compact connected Riemannian spin manifold with boundary and $\Phi\colon M\to\dX$ is a spin-structure preserving isometry such that $h=H_X\circ\Phi$.
Here $\dX$ carries the metric and spin structure induced by those of $X$ and $H_X$ is the unnormalized mean curvature of $\dX$ in $X$.
The sign convention is such that the boundary of an $(n+1)$-dimensional Euclidean ball of radius $R$ has mean curvature $\frac{n}{R}$. 

If the scalar curvature of $X$ is nonnegative, then we call $(X,\Phi)$ an \emph{NNSC spin fill-in} of $(M,h)$.
\end{definition}

Theorem~6 of \cite{HMZ} by Hijazi, Montiel, and Zhang says that if $(M,h)$ admits an NNSC spin fill-in then
\begin{equation}
\sqrt{\lambda_1(\D^2)} \ge \tfrac12 \min_M h.
\label{eq.HMZ}
\end{equation}

It has been pointed out that the original proof in \cite{HMZ} has a gap, see Remark~3.4 in \cite{CHZ}.
For the reader's convenience we provide a proof in the appendix, see Theorem~\ref{thm.HMZ}.
Combining this with the main theorem yields
$$
\min_M h \le \frac{n}{\RadSn(M)} .
$$

If we have equality in this estimate we must have equality in the main theorem, hence $M$ is isometric to a round sphere.
Moreover, we must have equality in \eqref{eq.HMZ}. 
By the equality part of Theorem~\ref{thm.HMZ}, any NNSC spin fill-in $X$ must be Ricci flat and $h$ must be constant (equal to $\frac{n}{2\RadSn(M)}$).

Theorems~4.1 and 4.2 in \cite{ST} by Shi and Tam then imply that $X$ is isometric to a Euclidean domain bounding a round sphere, hence to a Euclidean ball. 
We summarize:

\begin{theorem}[Gromov-Cecchini-Hirsch-Zeidler]
Let $M$ be an $n$-dimensional closed Riemannian spin manifold and let $h\colon M\to\R$ be a smooth function.
If $(M,h)$ admits an NNSC spin fill-in then
\begin{equation}
\min_M h \le \frac{n}{\RadSn(M)} .
\label{eq.GCHZ}
\end{equation}
If equality holds in \eqref{eq.GCHZ}, then any NNSC spin fill-in is isometric to a Euclidean ball. 
\hfill\qed
\end{theorem}

Inequality~\eqref{eq.GCHZ} can be found in Section~4.3 of Gromov's ``Four lectures on scalar curvature'' \cite{G4}.
The equality discussion has recently been carried out by Cecchini, Hirsch and Zeidler using different arguments in \cite{CHZ}*{Thm.~1.5}.

The remaining applications are new to the best of our knowledge.

\subsection{A conformal bound}
The positive lower bound on scalar curvature can be replaced by a positivity assumption on the Yamabe constant.

\begin{theorem}
Let $(M,g)$ be a closed connected Riemannian spin manifold of dimension $n\ge3$.
Suppose that the Yamabe constant $\Y(M,[g])$ of the conformal class of $g$ is positive.
Then
\begin{equation}
\RadSn(M)^2 \le \frac{n(n-1)\vol(M,g)^{2/n}}{\Y(M,[g])} .
\label{eq.RadConform}
\end{equation}
Equality holds if and only if $(M,g)$ is isometric to a round sphere.
\end{theorem}

\begin{proof}
Hijazi has shown in \cite{H}*{Thm.~1} that 
$$
\lambda_1(\D^2) \ge \frac{n}{4(n-1)}\frac{\Y(M,[g])}{\vol(M,g)^{2/n}} .
$$
Combining this with the main theorem yields \eqref{eq.RadConform}.

If we have equality in \eqref{eq.RadConform}, then we have equality in the main theorem and hence $(M,g)$ is isometric to a round sphere.

Conversely, if $(M,g)$ is isometric to a round sphere, we can arrange by rescaling that $M=\Sn$.
Then $\RadSn(M)=1$ and $\Y(M)=n(n-1)\vol(M)^{2/n}$.
Thus we have equality in \eqref{eq.RadConform}.
\end{proof}

% \subsection{Hyperspherical radius of $\boldsymbol{S^2}$ with an arbitrary metric}
% Next we consider $M=S^2$ equipped with an arbitrary Riemannian metric.
% It was shown by the author in Theorem~2 of \cite{B-Lower} that
% $$
% \lambda_1(\D^2) \ge \frac{4\pi}{\mathrm{area}(M)} .
% $$
% Combining the estimate with the main theorem yields:

% \begin{theorem}
% Let $g$ be an arbitrary metric on $S^2$.
% Then
% $$
% \mathsf{Rad}_{\mathbb{S}^2}(S^2,g)^2 \le \frac{\mathrm{area}(S^2,g)}{4\pi}
% $$
% with equality if and only if $g$ has constant curvature.
% \hfill\qed
% \end{theorem}

\subsection{Manifolds with special structures}
Suitable additional conditions on $M$ can lead to improvements in \eqref{eq.Friedrich} and hence in \eqref{eq.Llarull}.
We present three instances of this situation.
We start with Kähler manifolds.

\begin{theorem}
Let $M$ be a closed connected Kähler spin manifold of complex dimension $m\ge2$.
Assume its scalar curvature satisfies $\scal\ge s_0$ for some positive constant $s_0$.
Then
$$
\RadSn(M)^2 <
\begin{cases}
\frac{m^3}{m+1}\frac{4}{s_0}, \text{ if $m$ is odd},\\
m(m-1)\frac{4}{s_0}, \text{ if $m$ is even},\\
\frac{m^3}{m+2}\frac{4}{s_0}, \text{ if $m\ge4$ is even and $M$ is Kähler-Einstein}.
\end{cases}
$$
\end{theorem}

\begin{proof}
Kirchberg showed in \cite{K1}*{Thm.~A}, \cite{K2}*{Thm.~2} and in \cite{K3}*{Cor.~4.1} that 
$$
\lambda_1(\D^2) \ge 
\begin{cases}
\frac{m+1}{4m}s_0, \text{ if $m$ is odd},\\
\frac{m}{4(m-1)}s_0, \text{ if $m$ is even},\\
\frac{m+2}{4m}s_0, \text{ if $m\ge4$ is even and $M$ is Kähler-Einstein}.
\end{cases}
$$
Combining this with the main theorem yields the weak version of the inequality in the statement.
Equality cannot occur because $M$ would then have to be isometric to $S^n$ which is not a Kähler manifold.
Therefore, the inequalities are strict.
\end{proof}

Next we consider quaternionic Kähler manifolds.
Replacing Kirchberg's estimates by Kramer, Semmelmann and Weingart's estimate (\cite{KSW}*{Thm.~1.1})
$$
\lambda_1(\D^2) \ge \frac{n+12}{4(n+8)}s_0
$$
the same reasoning yields:

\usetagform{simple}
\begin{theorem}
Let $M$ be a closed connected quaternionic Kähler spin manifold of real dimension $n\ge8$.
Assume its scalar curvature satisfies $\scal\ge s_0$ for some positive constant $s_0$.
Then
\begin{equation}
\RadSn(M)^2 < \frac{n^2(n+8)}{n+12}\frac{1}{s_0}.
\tag{$\qed$}
\end{equation}
\end{theorem}
\usetagform{default}

The existence of a nontrivial harmonic $1$-form \emph{of constant length} also leads to an improvement.
Moroianu and Ornea showed in \cite{MO}*{Thm.~1.1} that in this case Friedrich's estimate \eqref{eq.Friedrich} can be improved to 
\begin{equation*}
\lambda_1(\D^2) \ge \frac14\frac{n-1}{n-2}s_0 .
\label{eq.MO}
\end{equation*}
Equality cannot occur in the main theorem because $S^n$ does not admit nontrivial harmonic $1$-forms.
This proves:

\usetagform{simple}
\begin{theorem}
Let $M$ be a closed connected Riemannian spin manifold of dimension $n\ge3$ which admits a nontrivial harmonic $1$-form of constant length.
Suppose that the scalar curvature of $M$ satisfies $\scal\ge s_0$ for some positive constant $s_0$.
Then
\begin{equation}
\RadSn(M)^2 < \frac{n^2(n-2)}{n-1}\frac{1}{s_0}.
\tag{$\qed$}
\end{equation}
\end{theorem}
\usetagform{default}

\section{Spinors with Lipschitz twists}
\label{sec.LipschitzTwists}

We now turn to the proof of the main theorem.
The current section provides some preparation for the proof which is then carried out in the next section.

For a linear map $A$ between finite-dimensional Euclidean vector spaces we denote by $|A|_1$ its trace norm, by $|A|_2$ its Frobenius norm and by $|A|_\infty$ its spectral norm.
In other words, if $0\le\mu_1\le\mu_2\le \dots \le \mu_n$ are the singular values of $A$, repeated according to multiplicity, then $|A|_1=\mu_1 + \dots + \mu_n$, $|A|_2=\sqrt{\mu_1^2 + \dots + \mu_n^2}$, and $|A|_\infty=\mu_n$.

Throughout this section, let $M$ be a closed Riemannian spin manifold of dimension $n\ge2$.
Let $f\colon M\to S^n$ be a Lipschitz map.
We denote the spinor bundle of $M$ by $\Sigma M$ and the one of $S^n$ by $\Sigma S^n$.
The spinor bundles are Hermitian vector bundles of rank $2^{\lfloor n/2\rfloor}$.
Let $\nn$ be the connection on $\Sigma S^n$ induced by the Levi-Civita connection.
For $s\in\R$ define the connection
\[
\ns_X\psi := \nn_X\psi + s X\cdot \psi 
\]
on $\Sigma S^n$.
Here $X\cdot\psi$ denotes the Clifford product of the tangent vector $X$ with the spinor $\psi$.
All connections $\ns$ are metric with respect to the natural Hermitian scalar product on $\Sigma S^n$.
Spinors which are parallel w.r.t.\ $\ns$ are known as \emph{Killing spinors}.
Nontrivial Killing spinors exist on $S^n$ for $s=\frac12$ and $s=-\frac12$.
In these cases they trivialize the spinor bundle.
More precisely, there exist Killing spinors for $s=\frac12$ which form an orthonormal basis of $\Sigma S^n$ at every point of $S^n$ and similarly for $s=-\frac12$.

We pull back $\Sigma S^n$ along $f$ to $M$ and obtain a Lipschitz bundle $f^*\Sigma S^n$, equipped with a Hermitian metric.
For each $s\in\R$ we pull back $\ns$ and obtain a metric connection $f^*\ns$ on $f^*\Sigma S^n$.
The corresponding twisted Dirac operator $\D_{f,s}$ is defined in the usual manner and yields a selfadjoint operator $H^1(M;\Sigma M\otimes f^*\Sigma S^n) \to L^2(M;\Sigma M\otimes f^*\Sigma S^n)$.
Here $L^2(M;\Sigma M\otimes f^*\Sigma S^n)$ and $H^1(M;\Sigma M\otimes f^*\Sigma S^n)$ are defined as the completions of the space of all Lipschitz sections of $\Sigma M\otimes f^*\Sigma S^n$ with respect to the usual $L^2$-norm $\|\cdot\|_{L^2}$ and Sobolev norm $\|\cdot\|_{L^2}+\|\nabla\cdot\|_{L^2}$, respectively.
See Section~4 of \cite{CHS} for further details, although many complications treated in \cite{CHS} do not arise in our setup because all our Riemannian metrics are smooth.

We denote by $\Ctriv$ the trivial $\C^{\spindim}$-bundle over $M$ with standard Hermitian scalar product and standard connection $d$.
The Dirac operators on $\Sigma M$ and on $\Sigma M\otimes \Ctriv$ will both be denoted by $\D$.
The have the same spectrum, up to multiplicity.

We choose $\nabla^{1/2}$-parallel spinor fields $(\psi_\alpha^+)_{\alpha=1,\dots,\spindim}$ which form an orthonormal basis of $\Sigma S^n$ at each point of $S^n$.
We pull back the $\psi_\alpha^+$ along $f$ and obtain sections $\Psi_\alpha^+:=f^*\psi_\alpha^+$ of $f^*\Sigma S^n$ which form an orthonormal basis at each point of $M$ and satisfy $(f^*\ns)_X\Psi_\alpha^+ = (s-\frac12) f^*(df(X)\cdot \psi_\alpha^+)$.

Each $\phi\in\Sigma M\otimes f^*\Sigma S^n$ can be expanded uniquely as
\begin{equation}
\phi = \sum_\alpha \phi_\alpha \otimes \Psi_\alpha^+ .
\label{eq.expandphi}
\end{equation}
The map $\phi\mapsto (\phi_\alpha)_{\alpha=1,\dots,2^{\lfloor\frac{n}{2}\rfloor}}$ defines a unitary vector bundle isomorphism 
\begin{equation}
U_+\colon \Sigma M\otimes f^*\Sigma S^n \to \Sigma M\otimes \Ctriv
\label{eq.defU+}
\end{equation}
of Lipschitz regularity.
We have $\|\phi\|_{L^2}^2 = \|U_+\phi\|_{L^2}^2 = \sum_\alpha\|\phi_\alpha\|_{L^2}^2$.
Moreover, $\phi\in H^1(M;\Sigma M\otimes f^*\Sigma S^n)$ if and only if $U_+\phi\in H^1(M;\Sigma M\otimes \Ctriv)$  which is equivalent to all $\phi_\alpha\in H^1(M;\Sigma M)$.

Similarly, we can use an orthonormal basis of $\nabla^{-1/2}$-parallel spinor fields $(\psi_\alpha^-)_{\alpha=1,\dots,\spindim}$.
Replacing $\Psi_\alpha^+$ in \eqref{eq.expandphi} by $\Psi_\alpha^- := f^*\phi_\alpha^-$ we obtain a second unitary vector bundle isomorphism $U_-\colon \Sigma M\otimes f^*\Sigma S^n \to \Sigma M\otimes \Ctriv$ with similar properties as $U_+$.

For $\alpha,\beta=1,\dots,\spindim$ let $\zeta_{\alpha\beta}^\pm$ be the complex $1$-forms defined by
\begin{equation}
\zeta_{\alpha\beta}^\pm(X) = \< df(X)\cdot \psi_\alpha^\pm,\psi_\beta^\pm\> .
\label{eq.defzeta}
\end{equation}
Since $f$ is Lipschitz, the differential $df$ exists almost everywhere on $M$ and defines an $L^\infty$-section.
Hence, $\zeta_{\alpha\beta}^\pm$ is an $L^\infty$-$1$-form on $M$.
Let $Z_{\alpha\beta}^\pm$ be the corresponding $L^\infty$-vector fields on $M$, i.e.,
\[
Z_{\alpha\beta} ^\pm
= 
\sum_{j=1}^n \zeta_{\alpha\beta}^\pm(e_j)e_j
=
\sum_{j=1}^n \< df(e_j)\cdot \psi_\alpha^\pm,\psi_\beta^\pm\>e_j,
\]
where $e_1,\dots,e_n$ is a local orthonormal tangent frame.
We obtain two $L^\infty$-endomorphism fields $Z^+$ and $Z^-$ on $\Sigma M\otimes \Ctriv$ by
$$
Z^\pm\colon (\phi_\alpha)_\alpha \mapsto \Big(\sum_\beta Z_{\beta\alpha}^\pm \cdot \phi_\beta\Big)_\alpha .
$$

\begin{lemma}\label{lem.DiracConjugate}
We have 
\begin{equation}
U_+ \circ \D_{f,s} \circ U_+^{-1} = \D + (s-\tfrac12)Z^+
\label{eq.DiracConjugate+}
\end{equation}
and 
\begin{equation}
U_- \circ \D_{f,s} \circ U_-^{-1} = \D + (s+\tfrac12)Z^- .
\label{eq.DiracConjugate-}
\end{equation}
\end{lemma}

\begin{proof}
We compute, using a local orthonormal tangent frame $e_1,\dots,e_n$:
\begin{align*}
\D_{f,s}\phi 
&=
\D_{f,s}\sum_\alpha \phi_\alpha \otimes \Psi_\alpha^\pm \\
&=
\sum_\alpha \D \phi_\alpha \otimes \Psi_\alpha^\pm + \sum_{j,\alpha} e_j\cdot \phi_\alpha \otimes \ns_{e_j}\Psi_\alpha^\pm \\
&=
\sum_\alpha \D \phi_\alpha \otimes \Psi_\alpha^\pm + \sum_{j,\alpha} e_j\cdot \phi_\alpha \otimes (s\mp\tfrac12)f^*(df(e_j)\cdot\psi_\alpha^\pm) \\
&=
\sum_\alpha \D \phi_\alpha \otimes \Psi_\alpha^\pm + (s\mp\tfrac12)\sum_{j,\alpha,\beta} e_j\cdot \phi_\alpha \otimes \<f^*(df(e_j)\cdot\psi_\alpha^\pm),\Psi_\beta^\pm\>\Psi_\beta^\pm \\
&=
\sum_\alpha \D \phi_\alpha \otimes \Psi_\alpha^\pm + (s\mp\tfrac12)\sum_{j,\alpha,\beta} e_j\cdot \phi_\alpha \otimes \<df(e_j)\cdot\psi_\alpha^\pm,\psi_\beta^\pm\>\Psi_\beta^\pm \\
&=
\sum_\alpha \Big(\D \phi_\alpha + (s\mp\tfrac12)\sum_{j,\beta} \<df(e_j)\cdot\psi_\beta^\pm,\psi_\alpha^\pm\>e_j\cdot \phi_\beta \Big)\otimes \Psi_\alpha^\pm \\
&=
\sum_\alpha \Big(\D \phi_\alpha + (s\mp\tfrac12)\sum_{\beta} Z_{\beta\alpha}^\pm\cdot \phi_\beta \Big)\otimes \Psi_\alpha^\pm .
\end{align*}
This proves the lemma.
\end{proof}

The lemma implies that $\D=U(\D+Z^-)U^{-1}$ where $U=U_+U_-^{-1}$.
In particular, $\D$ and $\D+Z^-$ have the same spectrum.
Similarly, $\D$ and $\D-Z^+$ are conjugate and hence have the same spectrum.

For $v,w\in T_pM$ we consider the endomorphism 
$$
\tau_{v,w}:=v\cdot w\otimes d_pf(v)\cdot d_pf(w)
$$
of $(\Sigma M\otimes f^*\Sigma S^n)_p$.
The following lemma provides the crucial estimate:

\begin{lemma}\label{lem.basicestimate}
Let $s\in\R$ and let $\phi \in H^1(M;\Sigma M\otimes f^*\Sigma S^n)$ with $\D_{f,s}\phi=0$.
We write $U_+\phi =: (\phi_\alpha)_\alpha$ if $s\ge0$ and $U_-\phi =: (\phi_\alpha)_\alpha$ if $s<0$.
Then
\begin{equation}
\sum_\alpha |\D\phi_\alpha|^2
\le
\big(\tfrac12-|s|\big)^2\, |df|_1^2 \,\sum_\alpha |\phi_\alpha|^2
\label{eq.basicestimate}
\end{equation}
holds almost everywhere on $M$.

Moreover, if $|s|\neq\frac12$, if $f$ is differentiable at $p\in M$, if $d_pf$ is a linear isometry and if equality holds in \eqref{eq.basicestimate} at $p$, then
\begin{equation}
\tau_{v,w}\,\phi(p) = \phi(p)
\label{eq.vwrigid}
\end{equation}
for all orthonormal $v,w\in T_pM$.
\end{lemma}

\begin{proof}
First note that since $\phi_\alpha$ and $D\phi_\alpha$ are $L^2$-spinors and $|df|_1$ is an $L^\infty$-function, both sides in \eqref{eq.basicestimate} are $L^1$-functions.
Hence, it makes sense to evaluate them almost everywhere.

Now we fix a point $p\in M$ at which the Lipschitz map $f$ is differentiable and choose orthonormal bases $e_1,\dots,e_n$ of $T_pM$ and $e_1^0,\dots,e_n^0$ of $T_{f(p)}S^n$ such that $df(e_j)=\mu_j e_j^0$ where $0\le\mu_1\le\dots\le\mu_n$ are the singular values of $df$ at~$p$.

W.l.o.g.\ we assume $s\ge0$, the case $s<0$ being entirely analogous.
From Lemma~\ref{lem.DiracConjugate} we have $(\D\phi_\alpha)_\alpha = (\frac12-s)Z^+(\phi_\alpha)_\alpha$.
We compute at~$p$:
\begin{align}
\sum_\alpha |\D \phi_\alpha|^2
&=
\big(\tfrac12-s\big)^2|Z^+(\phi_\alpha)_\alpha|^2 \\
&=
\big(\tfrac12-s\big)^2\sum_\alpha \Big\<\sum_\beta Z_{\beta\alpha}^+\phi_\beta, \sum_\gamma Z_{\gamma\alpha}^+\phi_\gamma\Big\> \notag\\
&=
\big(\tfrac12-s\big)^2\sum_{\alpha,\beta,\gamma,\delta} \Big\<\sum_\beta Z_{\beta\alpha}^+\phi_\beta, \sum_\gamma Z_{\gamma\delta}^+\phi_\gamma\Big\>\<\psi_\alpha^+,\psi_\delta^+\> \notag\\
&=
\big(\tfrac12-s\big)^2\sum_{\alpha,\beta,\gamma,\delta,j,k}\Big\< \<df(e_j)\cdot\Psi_\beta^+,\Psi_\alpha^+\>e_j\cdot \phi_\beta , \<df(e_k)\cdot\Psi_\gamma^+,\Psi_\delta^+\>e_k\cdot \phi_\gamma \Big\>\<\psi_\alpha^+,\psi_\delta^+\>\notag\\
&=
\big(\tfrac12-s\big)^2\sum_{\alpha,\beta,\gamma,\delta,j,k}\big\< e_j\cdot \phi_\beta , e_k\cdot \phi_\gamma \big\>\big\<\<df(e_j)\cdot\Psi_\beta^+,\Psi_\alpha^+\>\psi_\alpha^+,\<df(e_k)\cdot\Psi_\gamma^+,\Psi_\delta^+\>\psi_\delta^+\big\>\notag\\
&=
\big(\tfrac12-s\big)^2\sum_{\beta,\gamma,j,k}\big\< e_j\cdot \phi_\beta , e_k\cdot \phi_\gamma \big\>\big\<df(e_j)\cdot\Psi_\beta^+,df(e_k)\cdot\Psi_\gamma^+\big\>\notag\\
&=
\big(\tfrac12-s\big)^2\sum_{\beta,\gamma,j,k}\mu_j\mu_k\big\< e_j\cdot \phi_\beta , e_k\cdot \phi_\gamma \big\>\big\<e^0_j\cdot\Psi_\beta^+,e^0_k\cdot\Psi_\gamma^+\big\>\notag\\
&=
\big(\tfrac12-s\big)^2\sum_{\beta,\gamma,j,k}\mu_j\mu_k\big\< e_j\cdot e_k\cdot \phi_\beta , \phi_\gamma \big\>\big\<e^0_j\cdot e^0_k\cdot\Psi_\beta^+,\Psi_\gamma^+\big\>\notag\\
&=
\big(\tfrac12-s\big)^2\sum_{j,k}\mu_j\mu_k \big\<(e_j\cdot e_k\otimes e^0_j\cdot e^0_k)\phi,\phi\big\> .
\label{eq.sumDphialpha}
\end{align}
We investigate the summands of the right-hand side for fixed $j$ and $k$.
If $j=k$, then we have
\begin{align}
\big\<(e_j\cdot e_k\otimes e^0_j\cdot e^0_k)\phi,\phi\big\>
=
\big\<((-1)\otimes (-1))\phi,\phi\big\>
=
|\phi|^2.
\label{eq.j=k}
\end{align}
If $j\neq k$, then $e_j\cdot e_k\otimes e^0_j\cdot e^0_k$ is a selfadjoint involution.
Therefore,
\begin{align}
\big\<(e_j\cdot e_k\otimes e^0_j\cdot e^0_k)\phi,\phi\big\>
\le|\phi|^2
\label{eq.jneqk}
\end{align}
with equality if and only if 
\begin{align}
(e_j\cdot e_k\otimes e^0_j\cdot e^0_k)\phi=\phi .
\label{eq.jneqkrigid}
\end{align}
Inserting \eqref{eq.j=k} and \eqref{eq.jneqk} into \eqref{eq.sumDphialpha} yields
\[
\sum_\alpha |\D \phi_\alpha|^2
\le
\big(\tfrac12-s\big)^2\sum_{j,k}\mu_j\mu_k  |\phi|^2
=
\big(\tfrac12-s\big)^2 |df|_1^2 \sum_\alpha |\phi_\alpha|^2 .
\]
If $d_pf$ is an isometry, then we have $\mu_1=\dots=\mu_n=1$ at $p$ and the orthonormal basis $e_1,\dots,e_n$ can be chosen arbitrarily.
Together with \eqref{eq.jneqkrigid} this implies \eqref{eq.vwrigid}.
\end{proof}

Finally, we will need the following lemma:

\begin{lemma}\label{lem.Z+-}
We have 
$$
U_+^{-1}Z^+U_+ = U_-^{-1}Z^- U_- = \sum_j e_j \otimes df(e_j)
$$
where $e_1,\dots,e_n$ denotes a local orthonormal frame and the tangent vectors act by Clifford multiplication.
\end{lemma}

\begin{proof}
Writing $U_\pm\phi = (\phi_\alpha)_\alpha$ we compute:
\begin{align*}
U_\pm^{-1} Z^\pm U_\pm \phi
&=
U_\pm^{-1} Z^\pm (\phi_\alpha)_\alpha \\
&=
\sum_{\alpha\beta} Z_{\beta\alpha}^\pm \cdot\phi_\beta \otimes \Psi_\alpha^\pm\\
&=
\sum_{\alpha\beta j} \<df(e_j)\cdot\psi_\beta^\pm,\psi_\alpha^\pm\>e_j \cdot\phi_\beta \otimes \Psi_\alpha^\pm\\
&=
\sum_{\beta j} e_j \cdot\phi_\beta \otimes df(e_j)\cdot\Psi_\beta^\pm\\
&=
\sum_{j} (e_j \otimes df(e_j))\phi .
\qedhere
\end{align*}
\end{proof}

\section{Proof of the main theorem} 
\label{sec.spherrad} 

In this section we prove the main theorem.
Let $f\colon M\to S^n$ be a Lipschitz map with $\deg(f)\ne0$ and $\Lip(f)=\frac{1}{\RadSn(M)}$.

\subsection{\texorpdfstring{Proof of $\boldsymbol{\ker(\D_{f,0})\ne0}$ if the dimension $\boldsymbol{n}$ is even.}{Proof of ker(Df,0)≠0 if the dimension n is even}}
\label{subsec.nontrivialkernel}
In even dimension we have the chirality splittings of $\Sigma M=\Sigma^+M\oplus \Sigma^-M$ and $\Sigma S^n=\Sigma^+S^n\oplus \Sigma^-S^n$ at our disposal.
The connection $\nn$ preserves the splitting $\Sigma S^n=\Sigma^+S^n\oplus \Sigma^-S^n$ (unlike $\ns$ for $s\ne0$).
Thus, the twisted Dirac operator $\D_{f,0}$ can be restricted to four operators:
\begin{align*}
\D_{f,0}^{++}\colon& C^\infty(M;\Sigma^+ M\otimes f^*\Sigma^+ S^n)\to C^\infty(M;\Sigma^- M\otimes f^*\Sigma^+ S^n),\\
\D_{f,0}^{+-}\colon& C^\infty(M;\Sigma^+ M\otimes f^*\Sigma^- S^n)\to C^\infty(M;\Sigma^- M\otimes f^*\Sigma^- S^n),\\
\D_{f,0}^{-+}\colon& C^\infty(M;\Sigma^- M\otimes f^*\Sigma^+ S^n)\to C^\infty(M;\Sigma^+ M\otimes f^*\Sigma^+ S^n),\\
\D_{f,0}^{--}\colon& C^\infty(M;\Sigma^- M\otimes f^*\Sigma^- S^n)\to C^\infty(M;\Sigma^+ M\otimes f^*\Sigma^- S^n).
\end{align*}
The Atiyah-Singer index theorem applies (see Theorem~4.8 in \cite{CHS}) and yields
\begin{align*}
\ind(\D_{f,0}^{++}) + \ind(\D_{f,0}^{--})
&=
\ind(\D_{f,0}^{++}) - \ind(\D_{f,0}^{+-}) \\
&=
\<\Ahat(M)\cup f^*(\ch(\Sigma^+ S^n) - \ch(\Sigma^- S^n)),[M]\>\\
&=
\< \Ahat(M)\cup (-1)^{\frac{n}{2}}f^*\e(TS^n),[M]\>\\
&=
(-1)^{\frac{n}{2}}\< f^*\e(TS^n) ,[M]\>\\
&=
(-1)^{\frac{n}{2}}\deg(f)\<\e(TS^n),[S^n]\> \\
&= 
(-1)^{\frac{n}{2}}\cdot 2\cdot\deg(f).
\end{align*}
Here $\Ahat$ denotes the $\Ahat$-class, $\ch$ the Chern character and $\e$ the Euler class.
In the last line we used once more that $n$ is even and hence $S^n$ has Euler characteristic~$2$.
If $(-1)^{\frac{n}{2}}\cdot\deg(f)>0$ we find that $\ind(\D_{f,0}^{++})>0$ or $\ind(\D_{f,0}^{--})>0$.
If $(-1)^{\frac{n}{2}}\cdot\deg(f)<0$ we find that $\ind(\D_{f,0}^{++})<0$ or $\ind(\D_{f,0}^{--})<0$ and hence $\ind(\D_{f,0}^{-+})>0$ or $\ind(\D_{f,0}^{+-})>0$.
In any case, there exists a nontrivial $\phi\in H^1(M;\Sigma M\otimes f^*\Sigma S^n)$ with $\D_{f,0}\phi=0$.
We keep in mind, however, that $\phi$ is a section of one of the subbundles $\Sigma^\pm M\otimes f^*\Sigma^\pm S^n$.

\subsection{\texorpdfstring{Proof of $\boldsymbol{\ker(\D_{f,s})\ne0}$ for some $\boldsymbol{s\in[-\frac12,\frac12]}$ if the dimension $\boldsymbol{n}$ is odd.}{Proof of ker(Df,s)≠0 for some s∈[-½,½] if the dimension n is odd}}
\label{subsec.proof.spectralflow}
Denote the spectral flow of the operator family $[-\frac12,\frac12]\ni s\mapsto \D_{f,s}$ by $\sf(\D_{f,\cdot})$.
By Lemma~\ref{lem.DiracConjugate}, this is the same as the spectral flow of the family $[0,1]\ni s\mapsto \D + sZ^-$.

For smooth $f$, this spectral flow has been computed to $|\sf(\D_{f,\cdot})|=|\deg(f)|$ in \cite{LSW}*{Prop.~3.3}\footnote{Our connection $\ns$ corresponds to the connection $\nabla_{s+\frac12}=d+(s+\frac12)\omega$ in \cite{LSW}.} based on \cite{G}*{Thm.~2.8}.
To see that the result remains true for our Lipschitz map $f$, choose a smooth map $f_1\colon M\to S^n$ such that $\dist(f(x),f_1(x))<\pi$ for all $x\in M$.
Then we obtain a homotopy $[0,1]\ni t\mapsto f_t$ between $f=f_0$ and $f_1$ by joining the image points by the unique shortest geodesic in $\Sn$ connecting them.
In particular, $\deg(f)=\deg(f_1)$.
By homotopy invariance of the spectral flow, we have
\begin{align}
\sf\big([-\tfrac{1}{2},&\tfrac12]\ni s \mapsto \D_{f,s}\big) \notag\\
&=
\sf\big([0,1]\ni t\mapsto \D_{f_t,-1/2}\big)
+ \sf\big([-\tfrac{1}{2},\tfrac12]\ni s\mapsto \D_{f_1,s}\big)
- \sf\big([0,1]\ni t\mapsto \D_{f_t,1/2}\big) \label{eq.sfhomotopy}\\
&=
\sf\big([-\tfrac{1}{2},\tfrac12]\ni s\mapsto \D_{f_1,s}\big) . \notag
\end{align}
The negative sign of the last term in \eqref{eq.sfhomotopy} comes from the fact that we have to run that operator family backwards.
Moreover, $\sf\big([0,1]\ni t\mapsto \D_{f_t,-1/2}\big)=\sf\big([0,1]\ni t\mapsto \D_{f_t,1/2}\big)$ follows from $\D_{f_t,-1/2}$ and $\D_{f_t,1/2}$ having the same spectrum, cf.\ Section~\ref{sec.LipschitzTwists}.

Hence $\sf(\D_{f,\cdot})\neq0$ and therefore not all $\D_{f,s}$ can be invertible. 
Since the $\D_{f,s}$ are selfadjoint, at least one $\D_{f,s}$ must have a nontrivial kernel where $s\in[-\frac12,\frac12]$.

\subsection{\texorpdfstring{Proof of estimate \eqref{eq.DiracRad} in arbitrary dimension $\boldsymbol{n}$}{Proof of estimate (1) in arbitrary dimension n}}
\label{subsec.estimate}
Now we return to arbitrary dimension $n$.
Let $\phi\in\ker(\D_{f,s})\ne0$ for some $s\in[-\frac12,\frac12]$.
We put $(\phi_\alpha)_\alpha:=U_+\phi$ if $s\ge0$ and $(\phi_\alpha)_\alpha:=U_-\phi$ if $s<0$.
Then we have $\phi_\alpha\in H^1(M;\Sigma M)$, not all $\phi_\alpha$ are zero, and, by Lemma~\ref{lem.basicestimate},
\begin{equation}
\sum_\alpha |\D\phi_\alpha|^2
\le
\big(\tfrac12-|s|\big)^2|df|_1^2 \,\sum_\alpha |\phi_\alpha|^2
\le
\tfrac14 |df|_1^2 \,\sum_\alpha |\phi_\alpha|^2 .
\label{eq.lambdaRadPointwise}
\end{equation}
Since $\Lip(f)=1/\RadSn(M)$ all singular values of $df$ are bounded by $1/\RadSn(M)$ and hence $|df|_1\le n/\RadSn(M)$.
Integrating over $M$ we find
\begin{equation}
\lambda_1(\D^2)\sum_\alpha \|\phi_\alpha\|_{L^2}^2
\le
\sum_\alpha \|\D \phi_\alpha\|_{L^2}^2
\le
\tfrac{n^2}{4\RadSn(M)^2} \sum_\alpha \|\phi_\alpha\|_{L^2}^2 .
\label{eq.lambdaRad}
\end{equation}
This implies $\lambda_1(\D^2) \le \frac{n^2}{4\RadSn(M)^2}$ and concludes the proof of \eqref{eq.DiracRad}.

\subsection{Equality for the sphere.}
If $M=\Sn(R)$, then the smallest Dirac eigenvalue is known to be $\lambda_1(\D^2)=\frac{n^2}{4R^2}$, see e.g.\ Theorem~1 in \cite{B-SpaceForms}.
Thus equality holds in \eqref{eq.DiracRad} in this case.

\subsection{\texorpdfstring{Beginning of the equality discussion in arbitrary dimension $\boldsymbol{n}$.}{Beginning of the equality discussion in arbitrary dimension n}}
Conversely, assume equality in \eqref{eq.DiracRad}.
For notational convenience, we assume that $\RadSn(M)=1$ which can be achieved by rescaling the metric.
Hence, $\Lip(f)=1$ and thus the singular values of $df$ satisfy $\mu_j\le1$ almost everywhere.

Let $\phi=\sum_\alpha \phi_\alpha\otimes \Psi_\alpha\in H^1(M;\Sigma M\otimes f^*\Sigma S^n)$ be nontrivial with $\D_{f,s}\phi=0$ with $s\in[-\frac12,\frac12]$.
Equality in \eqref{eq.DiracRad} implies equality in \eqref{eq.lambdaRad} and thus $\phi_\alpha$ are eigenspinors for $\D^2$ for the eigenvalue $\lambda_1(\D^2)$.
In particular, they are smooth and the zero locus of the nontrivial $\phi_\alpha$ have Hausdorff codimension $\ge2$ by the main theorem in \cite{B-Nodal}.
Thus, $\phi$ is continuous and its zero locus $\ZZ$ has Hausdorff codimension~$\ge2$.
In particular, it is a nullset and its complement is open and connected.

Moreover, equality in \eqref{eq.DiracRad} implies equality in $|df|_1\le n/\RadSn(M)=n$ unless $\phi=0$.
Therefore, all $\mu_j=1$.
This means that $df$ is a linear isometry almost everywhere.
It remains to show that $d_pf$ preserves the orientation for almost all $p$ or it reverses the orientation for almost all $p$ since it then follows from Theorem~1 in \cite{KMS} that $f\colon M\to S^n$ is a smooth Riemannian covering.
Since $M$ is connected and $S^n$ is simply connected, $f$ then is an isometry.

Equality in \eqref{eq.lambdaRad} implies equality in \eqref{eq.lambdaRadPointwise}.
Thus $s=0$ and, by Lemma~\ref{lem.basicestimate}, equation \eqref{eq.basicestimate} holds almost everywhere.

Let $\omega_\C$ be the Clifford volume element for $\Sigma M$.
It is a field of selfadjoint involutions on $\Sigma M$ and can locally be written as $\omega_\C = i^{\lfloor \frac{n+1}{2}\rfloor}e_1\cdots e_n$ where $e_1,\dots,e_n$ is a positively oriented local orthonormal tangent frame.
Similarly, let $\omega_\C^0$ be the Clifford volume element for $\Sigma S^n$.

\subsection{\texorpdfstring{Conclusion of the equality discussion for even dimension $\boldsymbol{n}$.}{Conclusion of the equality discussion for even dimension n}}
\label{subsec.eqneven}
We conclude the equality discussion in \eqref{eq.DiracRad} if the dimension $n$ is even.
The volume element $\omega_\C$ acts as $\pm 1$ on $\Sigma^\pm M$ and similarly for $\omega_\C^0$.
Since $\phi$ is a section of one of the subbundles $\Sigma^\pm M\otimes f^*\Sigma^\pm S^n$, we have that 
\begin{equation}
(\omega_\C\otimes \omega_\C^0)\phi=\eps \phi
\label{eq.evenn.phi1}
\end{equation}
for a constant sign $\eps=\pm 1$.

Let $p\in M$ such that $f$ is differentiable at $p$.
Let $e_1,\dots,e_n$ be a positively oriented orthonormal basis of $T_pM$.
Since $d_pf$ is a linear isometry, $d_pf(e_1),\dots,d_pf(e_n)$ is an orthonormal basis of $T_{f(p)}S^n$.
It is positively oriented if and only if $d_pf$ is orientation preserving. 
Therefore,
$$
(-1)^{\lfloor \frac{n+1}{2}\rfloor}\omega_\C \otimes \omega_\C^0 
= 
\sigma(p) e_1\cdots e_n\otimes d_pf(e_1) \otimes d_pf(e_n)
=
\sigma(p) \tau_{e_1,e_2}\circ \tau_{e_3,e_4}\circ \dots \circ \tau_{e_{n-1,n}}
$$
where $\sigma(p)=+1$ if $d_pf$ is orientation preserving and $\sigma=-1$ if $d_pf$ reverses the orientation.
From \eqref{eq.vwrigid} we find
\begin{equation}
(-1)^{\lfloor \frac{n+1}{2}\rfloor}(\omega_\C\otimes \omega_\C^0)\phi(p) = \sigma(p) \phi(p).
\label{eq.evenn.phi2}
\end{equation}
Comparing \eqref{eq.evenn.phi1} and \eqref{eq.evenn.phi2} we find $\sigma=(-1)^{\lfloor \frac{n+1}{2}\rfloor}\eps$ almost everywhere.
Therefore, $d_pf$ preserves the orientation for almost all $p$ or it reverses the orientation for almost all $p$.
This concludes the proof for even $n$.

\subsection{\texorpdfstring{Conclusion of the equality discussion for odd dimension $\boldsymbol{n}$.}{Conclusion of the equality discussion for odd dimension n}}
Finally, assume equality in \eqref{eq.DiracRad} and that the dimension $n$ is odd.
In this case $\omega_\C$ and $\omega_C^0$ act as the identity on the spinor bundles. 
For $p$, $e_1,\dots,e_n$, and $\sigma$ as above we get 
\begin{align*}
(-1)^{\lfloor \frac{n+1}{2}\rfloor}\phi(p)
&=
(-1)^{\lfloor \frac{n+1}{2}\rfloor}(\omega_\C \otimes \omega_\C^0) \phi(p) \\
&=
\sigma(p)(e_1\otimes d_pf(e_1))\tau_{e_2,e_3}\circ \dots \circ \tau_{e_{n-1,n}}\phi(p)\\
&=
\sigma(p)(e_1\otimes d_pf(e_1))\phi(p) .
\end{align*}
Since $e_1$ can be chosen arbitrarily, this yields 
$$
(v\otimes d_pf(v))\phi(p)
=
(-1)^{\lfloor \frac{n+1}{2}\rfloor}\sigma(p)\phi(p)
$$
for all unit vectors $v\in T_pM$.
Lemma~\ref{lem.Z+-} implies $U_+^{-1}Z^+U_+\phi = (-1)^{\lfloor \frac{n+1}{2}\rfloor}n\sigma\phi$.
From Lemma~\ref{lem.DiracConjugate} we get
\begin{align*}
\D U_+\phi
&=
U_+\D_{f,0}\phi + \tfrac12 Z^+U_+\phi 
=
0 + \tfrac12 U_+\big((-1)^{\lfloor \frac{n+1}{2}\rfloor}n\sigma\phi\big) 
=
(-1)^{\lfloor \frac{n+1}{2}\rfloor}\tfrac{n\sigma}{2} U_+\phi .
\end{align*}
Writing $U_+\phi=(\phi_\alpha)_\alpha$ this implies for each $\alpha$:
\begin{equation}
\D \phi_\alpha = (-1)^{\lfloor \frac{n+1}{2}\rfloor}\tfrac{n\sigma}{2}\phi_\alpha .
\label{eq.fakeeigen}
\end{equation}
Therefore,
\begin{align*}
\sum_\alpha |\D\phi_\alpha|^2
=
\tfrac{n^2}{4}\sum_\alpha |\phi_\alpha|^2
=
\lambda_1(\D^2)\sum_\alpha |\phi_\alpha|^2
\end{align*}
Integration over $M$ yields
$$
\sum_\alpha \|\D\phi_\alpha\|_{L^2}^2
=
\lambda_1(\D^2)\sum_\alpha \|\phi_\alpha\|_{L^2}^2
$$
and hence $\D^2\phi_\alpha = \lambda_1(\D^2)\phi_\alpha$.
In particular, all $\phi_\alpha$ are smooth. 
Thus, \eqref{eq.fakeeigen} implies that $\sigma$ is smooth and hence locally constant on $M\setminus\ZZ$.
Since $M\setminus\ZZ$ is connected, $\sigma$ is constant which concludes the proof.
\hfill\qed

\section{More general target spaces}\label{sec.X}
We can allow more general target spaces than just the standard sphere $\Sn$.
At least if the dimension $n$ is even (which we assume for the rest of this section), a good class of targets seems to be the following.

\begin{definition}
Let $\XX$ be a connected closed $n$-dimensional Riemannian spin manifold with Euler number $\chi(\XX)\ne0$.
Let $W$ be a smooth field of positive semidefinite symmetric endomorphisms of the tangent bundle $T\XX$.
The pair $(\XX,W)$ is called an \emph{admissible comparison space} if the spinor bundle $\Sigma \XX$ of $\XX$ can be trivialized by spinors $\psi$ satisfying 
\begin{equation}
\nabla_X\psi = \tfrac12 W(X)\cdot \psi
\label{eq.KillingW}
\end{equation}
as well as by spinors $\psi$ satisfying 
\begin{equation}
\nabla_X\psi = -\tfrac12 W(X)\cdot \psi .
\label{eq.Killing-W}
\end{equation}
\end{definition}

Spinors satisfying \eqref{eq.KillingW} or \eqref{eq.Killing-W} are known as \emph{generalized Killing spinors}.

\begin{example}
The admissible comparison space we have been using throughout this article is $(\XX,W)=(\Sn,\id_{T\Sn})$.
\end{example}

\begin{example}\label{ex.convexdomain}
More generally, let $\Omega\subset\R^{n+1}$ be a convex domain with smooth boundary~$\XX$.
The boundary inherits a Riemannian metric and a spin structure.
Let $\nu$ be the exterior unit normal vector field along $\XX$ and let $W(X)=-\nabla_X\nu$ the Weingarten map.
Then $W$ is a positive semidefinite symmetric endomorphism field.

Since $n$ is even, the restriction of the spinor bundle $\Sigma\R^{n+1}$ to $\XX$ can be naturally identified with $\Sigma\XX$.
We can trivialize $\Sigma\R^{n+1}$ or $\Sigma^+\R^{n+1}$ by parallel spinors.
Restricting a parallel spinor $\psi$ to $\XX$, the Gauss equation for the spinorial connection (see e.g.\ Proposition~2.1 in \cite{B-harmonic}) yields
$$
0 = \nabla^{\R^{n+1}}_X\psi 
= \nabla_X\psi + \tfrac12 W(X)\star\nu\star\psi 
= \nabla_X\psi + \tfrac12 W(X)\cdot\psi .
$$
Here "$\star$" denotes Clifford multiplication w.r.t.\ $\R^{n+1}$ and "$\cdot$" is Clifford multiplication for $\XX$.
Thus, we have a trivialization of $\Sigma\XX$ by spinors solving \eqref{eq.Killing-W}.
If $\psi$ solves \eqref{eq.Killing-W}, we can put $\phi=\nu\star\psi$.
We then get
\begin{align*}
\nabla_X\phi
&=
\nabla^{\R^{n+1}}_X\phi - \tfrac12 W(X)\cdot\phi 
=
\nabla^{\R^{n+1}}_X(\nu\star\psi) - \tfrac12 W(X)\cdot \phi \\
&=
\nabla_X\nu\star\psi + \nu\star\nabla^{\R^{n+1}}_X\psi - \tfrac12 W(X)\cdot \phi
=
-W(X)\star\psi + 0 - \tfrac12 W(X)\cdot \phi \\
&=
W(X)\star\nu\star\phi - \tfrac12 W(X)\cdot \phi
=
\tfrac12 W(X)\cdot \phi .
\end{align*}
Thus, we also have a trivialization of $\Sigma\XX$ by spinors solving \eqref{eq.KillingW}.
This shows that $(\XX,W)$ is an admissible comparison space.
\end{example}

\begin{definition}\label{def.hyperX}
Let $(\XX,W)$ be an admissible comparison space and $M$ be a closed connected Riemannian manifold with $\dim(\XX)=\dim(M)=n$.
The \emph{hyper-$\XX$-radius} $\RadX(M)$ of $M$ is the supremum of all numbers $R>0$ such that there exists a Lipschitz map $f\colon M\to \XX$ with Lipschitz constant $\Lip(f)\le 1/R$ and $\deg(f)\ne0$.
\end{definition}

For $\XX=\Sn$ and $W=\id_{T\Sn}$ we recover the hyperspherical radius of $M$.
The same reasoning as in Section~\ref{sec.HypersphericalRadius} yields:

\begin{enumerate}[\myicon]
\item
\listdisplay{\RadX(M)\le \frac{\diam(M)}{\diam(\XX)},}
\item
\listdisplay{\RadX(M)^n \le \frac{\vol(M)}{\vol(\XX)}}
%\label{eq.VolumeEstimateX}

\noindent
with equality if and only if $M$ is isometric to $\XX$ up to rescaling to the metric,
\item
\listdisplay{\RadX(M,\lambda^2g) = \lambda\cdot\RadX(M,g),}
\item
\listdisplay{\mathsf{Rad}_{(\XX,\lambda^2g_\XX)}(M) = \tfrac{1}{\lambda}\mathsf{Rad}_{(\XX,g_\XX)}(M),}
\item
\listdisplay{\RadX(\XX)=1,}
\item 
The supremum in Definition~\ref{def.hyperX} is a maximum.
In other words, there exists a Lipschitz map $f\colon M\to \XX$ with $\Lip(f)= 1/\RadX(M)$ and $\deg(f)\neq0$.
\end{enumerate}

\begin{theorem}\label{thm.DiracRadX}
Let $(\XX,W)$ be an admissible comparison space and let $M$ be a connected closed Riemannian spin manifold with even dimension $\dim(M)=\dim(\XX)\ge2$.
Denote the Dirac operator acting on spinor fields of $M$ by $\D$.
Then
\begin{equation}
\lambda_1(\D^2) \le \frac{\max_M \tr(W)^2}{4\,\RadX(M)^2}.
\label{eq.DiracRadX}
\end{equation}
If equality holds in \eqref{eq.DiracRadX} and $W$ is positive definite almost everywhere, then $\tr(W)$ is constant and there is a finite Riemannian (up to rescaling the metric) covering map $M\to\XX$.
\end{theorem}

We first prove an auxiliary lemma.
See also Lemma~B.1 in \cite{CHZ} for a similar statement.

\begin{lemma}
\label{lem.normcomposition}
Let $U$ and $V$ be $n$-dimensional Euclidean vector spaces, let $F\colon U\to V$ be linear and let $W\colon V\to V$ be symmetric and positive semidefinite.
Then
\begin{equation}
|W\circ F|_1 \le \tr(W) |F|_\infty.
\label{eq.normcomposition}
\end{equation}
If $W$ is positive definite and equality holds in \eqref{eq.normcomposition}, then $F$ is $|F|_\infty$ times an isometry.
\end{lemma}

\begin{proof}
If $F=0$, then inequality~\eqref{eq.normcomposition} is trivial.
So assume $F\neq 0$.
Let $e_1,\dots,e_n$ be an orthonormal basis of $U$ and $e_1^0,\dots,e_n^0$ one of $V$ such that $W(F(e_j))=\mu_j e_j^0$ where $\mu_j$ are the singular values of $W\circ F$.
Since $W$ is positive semidefinite it has a positive semidefinite symmetric square root $L\colon V\to V$.
We compute
\begin{align*}
|W\circ F|_1 
&=
\sum_j \mu_j \\
&=
\sum_j \< W(F(e_j)), e_j^0\>\\
&=
\sum_j \< L(F(e_j)), L(e_j^0)\>\\
&=
\sum_{jk} \< L(F(e_j)),e_k^0\>\<e_k^0, L(e_j^0)\>\\
&=
\sum_{jk} \frac{1}{\sqrt{|F^*|_\infty}}\< L(F(e_j)),e_k^0\>\cdot\sqrt{|F^*|_\infty}\<e_k^0, L(e_j^0)\>\\
&\le
\tfrac12\sum_{jk}\Big( \frac{1}{|F^*|_\infty}\< L(F(e_j)),e_k^0\>^2 + |F^*|_\infty\<e_k^0, L(e_j^0)\>^2\Big)\\
&=
\frac{1}{2|F^*|_\infty}\sum_{jk}\< e_j,F^*(L(e_k^0))\>^2 + \frac{|F^*|_\infty}{2}\sum_j|L(e_j^0)|^2\\
&=
\frac{1}{2|F^*|_\infty}\sum_{k}|F^*(L(e_k^0))|^2 + \frac{|F^*|_\infty}{2}|L|_2^2\\
&\le
\frac{|F^*|_\infty}{2}\sum_{k}|L(e_k^0)|^2 + \frac{|F^*|_\infty}{2}|L|_2^2\\
&=
|F^*|_\infty\cdot |L|_2^2\\
&=
|F|_\infty\cdot \tr(W).
\end{align*}
This proves \eqref{eq.normcomposition}.

Now assume that $W$ is positive definite and that equality holds in \eqref{eq.normcomposition}.
Again, we can assume $F\neq 0$.
Then, in particular, equality must hold in the first inequality of the above computation.
This means that 
$$
\sqrt{|F^*|_\infty}\<e_k^0, L(e_j^0)\> = \frac{1}{\sqrt{|F^*|_\infty}}\< L(F(e_j)),e_k^0\>
$$
for all $j$ and $k$.
Thus 
$$
|F^*|_\infty\<e_k^0, L(e_j^0)\> = \< L(F(e_j)),e_k^0\>
$$
for all $j$ and $k$ and hence
$$
|F^*|_\infty L(e_j^0) = L(F(e_j))
$$
for all $j$.
Composing both sides with $L$ yields
$$
|F^*|_\infty W(e_j^0)=W(F(e_j)) = \mu_j e_j^0.
$$
Thus, $e_1^0,\dots,e_n^0$ form an eigenbasis for $W$ for the eigenvalues $\frac{\mu_1}{|F^*|_\infty},\dots,\frac{\mu_n}{|F^*|_\infty}$.
Since $W$ is positive definite, all these eigenvalues are positive, and we can compute
\begin{align*}
F(e_j) 
&= 
W^{-1}\circ W\circ F(e_j) 
= 
\frac{|F^*|_\infty}{\mu_j}\mu_j e_j^0
=
|F^*|_\infty e_j^0
=
|F|_\infty e_j^0 .
\end{align*}
Therefore, all singular values of $F$ are equal to $|F|_\infty$ and hence $F$ is $|F|_\infty$ times an isometry.
\end{proof}

\begin{proof}[Proof of Theorem~\ref{thm.DiracRadX}]
We indicate the necessary changes to the proof of the main theorem.
We can now define the unitary vector bundle isomorphisms $U_\pm\colon \Sigma M\otimes f^*\Sigma\XX \to \Sigma M\otimes \Ctriv$ in \eqref{eq.defU+} using the generalized Killing spinors on $\XX$ instead of the Killing spinors on $\Sn$.
If we replace $df$ by $W\circ df$ in the definition of the $1$-forms $\zeta_{\alpha\beta}^{\pm}$ and in Lemmas~\ref{lem.DiracConjugate}, \ref{lem.basicestimate}, and \ref{lem.Z+-}, then these lemmas still hold true.

The first step~\ref{subsec.nontrivialkernel} of the proof of the main theorem still works because we have assumed that $\XX$ has nontrivial Euler number which insures nontriviality of the index.
Proof step~\ref{subsec.estimate} also works and yields a nontrivial tuple of spinor fields $\phi_\alpha\in H^1(M;\Sigma M)$ satisfying
\begin{equation*}
\sum_\alpha |\D\phi_\alpha|^2
\le
\tfrac14 |W\circ df|_1^2 \,\sum_\alpha |\phi_\alpha|^2 .
%\label{eq.lambdaRadPointwiseW}
\end{equation*}
Lemma~\ref{lem.normcomposition} implies 
\begin{equation}
\sum_\alpha |\D\phi_\alpha|^2
\le
\tfrac14 \tr(W)^2|df|_\infty^2 \,\sum_\alpha |\phi_\alpha|^2 
\le
\tfrac14 \tr(W)^2\Lip(f)^2 \,\sum_\alpha |\phi_\alpha|^2 .
\label{eq.lambdaRadPointwiseW}
\end{equation}
This proves \eqref{eq.DiracRadX}.
% \begin{equation}
% \lambda_1(\D^2) \le \frac{\max_M \tr(W)^2}{4\,\RadX(M)^2}.
% \label{eq.DiracRadX}
% \end{equation}

Now assume that equality holds in \eqref{eq.DiracRadX} and that $W$ is almost everywhere positive definite.
Then we have equality in \eqref{eq.lambdaRadPointwiseW} and hence $|df|_{\infty}=\Lip(f)$ almost everywhere. 
From the equality case in Lemma~\ref{lem.normcomposition} we deduce that $df$ is $\Lip(f)$ times a linear isometry almost everywhere.

As in the proof of Lemma~\ref{lem.basicestimate} find that \eqref{eq.jneqkrigid} holds.
This implies $\tau_{v,w}\,\phi(p) = \Lip(f)^2\phi(p)$ for almost all $p\in M$ and orthonormal tangent vectors $v,w\in T_pM$.
Now the argument in Section~\ref{subsec.eqneven} applies and shows that $f\colon M\to \XX$ is a Riemannian covering (up to scaling by the factor $\Lip(f)$).
Finally, equality in \eqref{eq.DiracRadX} implies that $\tr(W)=\max_M\tr(W)$ is constant. 
\end{proof}

\begin{example}
Consider the situation described in Example~\ref{ex.convexdomain}, i.e., $\XX=\partial\Omega$ where $\Omega\subset\R^{n+1}$ is a convex domain with smooth boundary. 
We assume that the Weingarten map $W$ of the boundary of $\Omega$ is positive definite almost everywhere.
If we now have equality in \eqref{eq.DiracRadX}, then $\XX$ has constant mean curvature and must be a round sphere by Alexandrov's theorem.
Since the sphere is simply connected, the Riemannian covering $M\to\XX$ must be an isometry (up to scaling of the metric).
Thus, $M$ is a round sphere as well.

As an application of these considerations let $h\in C^\infty(\XX)$ be such that $\min_\XX h \ge \max_\XX \tr(W)$.
Assume that there is an NNSC spin fill-in of $(\XX,h)$.
As in Section~\ref{subsec.fillin} we find
$$
\max_\XX \tr(W) \le \min_\XX h \le 2\sqrt{\lambda_1(\D^2)}\le \frac{\max_\XX\tr(W)}{\RadX(\XX)} = \max_\XX \tr(W).
$$
Thus, we have equality in \eqref{eq.DiracRadX} and hence $\XX$ is a round sphere.
Reasoning as in Section~\ref{subsec.fillin}, we conclude that the spin fill-in is a Euclidean ball.
\end{example}

\section{Concluding remarks and questions}
\label{sec.conclusion}

We conclude with a few remarks and some related questions.

\subsection{Alternative approach to the odd-dimensional case}
There is an alternative way to prove the estimate \eqref{eq.DiracRad} for odd dimension $n$, completely avoiding the spectral flow argument in Section~\ref{subsec.proof.spectralflow}.
This alternative argument is more in the spirit of Llarull's way to treat the odd-dimensional case in \cite{Ll}.

Namely, let $\rho>0$ and put $\tilde{M} :=M\times \mathbb{S}^1(\rho)$ and where we equip $\mathbb{S}^1(\rho)$ with the spin structure for which it has a parallel spinor.
Then $\lambda_1(\tilde{\D}^2)=\lambda_1(\D^2)$ where $\tilde{\D}$ is the Dirac operator of $\tilde{M}$.

The singular values of the differential of $f\times \id\colon \tilde{M} =M\times \mathbb{S}^1(\rho) \to S^n\times \mathbb{S}^1(2)$ are same as those for $f$ plus one singular value $\frac{2}{\rho}$ coming from the $S^1$-direction.

There exists a smooth map $h\colon S^n\times S^1(2)\to S^{n+1}$ with Lipschitz constant $1$ and of degree~$1$, see Lemma~2.3 in \cite{BBHW}.
Thus, the Lipschitz map $F:=h\circ(f\times\id)\colon \tilde{M}\to S^{n+1}$ has degree $1$ and the singular values of its differential satisfy $0\le \mu_j\le \frac{1}{R}$ for $j=1,\dots,n$ and $\mu_{n+1}\le\frac{2}{\rho}$.
The discussion of the even-dimensional case now gives
\begin{align*}
\sum_\alpha \|\tilde{\D} \phi_\alpha\|_{L^2}^2
&\le
\tfrac14|dF|_1^2 \sum_\alpha \|\phi_\alpha\|_{L^2}^2 \\
&\le
\tfrac14(|df|_1+\tfrac{2}{\rho})^2 \sum_\alpha \|\phi_\alpha\|_{L^2}^2 \\
&\le
\tfrac14(\tfrac{n}{\RadSn(M)}+\tfrac{2}{\rho})^2 \sum_\alpha \|\phi_\alpha\|_{L^2}^2 .
\end{align*}
We find
$$
\lambda_1(\D^2)
=
\lambda_1(\tilde{\D}^2)
\le
\tfrac14(\tfrac{n}{\RadSn(M)}+\tfrac{2}{\rho})^2 .
$$
Taking the limit $\rho\nearrow \infty$ concludes the proof of \eqref{eq.DiracRad} for odd $n$.

\begin{question}
Can the equality discussion be carried out using this approach?
\end{question}

\subsection{A refined hyperspherical radius}

In Section~\ref{subsec.estimate} it has been shown that 
$$
\lambda_1(\D^2) \le \tfrac14 |df|_1^2 
$$
for any Lipschitz map $f\colon M\to\Sn$.
Together with $|df|_1 \le n\,\Lip(f)$ this yields the estimate~\eqref{eq.DiracRad}.

We can introduce a refined hyperspherical radius by letting $\RadSn^\tr(M)$ to be the supremum of all numbers $R>0$ such that there exists a Lipschitz map $f\colon M\to \Sn$ with $|df|_1\le n/R$ and $\deg(f)\neq0$.
Then $\RadSn^\tr(M)\ge\RadSn(M)$ and we have actually proved the potentially stronger estimate
\begin{equation}
\lambda_1(\D^2) \le \frac{n^2}{4\,\RadSn^\tr(M)^2}.
\label{eq.DiracRadRefined}
\end{equation}
\begin{question}
What is the equality case in \eqref{eq.DiracRadRefined}?
\end{question}

There may be a relation to Theorem~3 in \cite{Li}.

\subsection{\texorpdfstring{$\boldsymbol{\Lambda^2}$-contracting maps}{Λ²-contracting maps}}
Another way to modify the hyperspherical radius would be to let $\RadSn^{\Lambda^2}(M)$ be the supremum of all numbers $R>0$ such that there exists a Lipschitz map $f\colon M\to \Sn$ with $|\Lambda^2df|_\infty\le 1/R^2$ and $\deg(f)\neq0$.
Here $\Lambda^2df\colon \Lambda^2T_pM \to \Lambda^2 T_{f(p)\Sn}$ is the induced linear map on $2$-vectors and $|\cdot|_\infty$ denotes the spectral norm.
Since $|\Lambda^2df|_\infty\le|df|_\infty^2$, we have $\RadSn^{\Lambda^2}(M)\ge\RadSn(M)$.

\begin{question}
Does the estimate
$$
\lambda_1(\D^2) \le \tfrac14 |\Lambda^2 df|_\infty
$$
hold and does equality imply that $f$ is a Riemannian isometry?

If true, we would have
$$
\lambda_1(\D^2) \le \frac{n^2}{4\,\RadSn^{\Lambda^2}(M)^2}
$$
and the discussion in Section~\ref{subsec.appl.Llarull} would yield Theorem~C in \cite{Ll} which is stronger than Theorem~B.
\end{question}

\subsection{More general target spaces}
In Section~\ref{sec.X} we have discussed how one can replace $\Sn$ by more general comparison spaces $\XX$ in even dimensions.

\begin{question}
What would be a good class of admissible comparison spaces in odd dimensions for which Theorem~\ref{thm.DiracRadX} still holds?
\end{question}

\appendix

\section{Mean curvature of NNSC spin fill-ins}
\label{app.HMZ}

In this appendix we provide a proof of Theorem~6 in \cite{HMZ} for the reader's convenience as the proof in \cite{HMZ} has a gap.
There is no claim of originality.

Let $A$ be a (generally unbounded) selfadjoint operator in a Hilbert space $\HH$ with discrete spectrum.
We denote by $\lambda_1(A)$ its smallest nonnegative eigenvalue.
For any real interval $I\subset\R$ we denote by $E(A,I)\subset\HH$ the sum of all eigenspaces of $A$ to eigenvalues which lie in~$I$.

\begin{theorem}[Hijazi-Montiel-Zhang]\label{thm.HMZ}
Let $M$ be an $n$-dimensional connected closed Riemannian spin manifold.
Denote the Dirac operator acting on spinor fields of $M$ by $\D$.

Then for any NNPS fill-in $(X,h)$ of $M$ the following estimate holds:
\begin{equation}
\min_M h \le 2\sqrt{\lambda_1(\D^2)}.
\label{eq.meanDirac}
\end{equation}
In the case of equality in \eqref{eq.meanDirac}, $h$ is constant and $X$ is Ricci flat.
\end{theorem}

\begin{proof}
For notational simplicity we will identify $\dX$ with $M$.
Let $\D^X$ be the Dirac operator acting on sections of the spinor bundle $\Sigma X$ of $X$.
If $n$ is even, then $\Sigma X|_\dX$ can be naturally identified with the spinor bundle $\Sigma\dX=\Sigma M$ of $M$ and the intrinsic Dirac operator $\D=:\D^\dX$ of $M$ is an adapted boundary operator for $D^X$.
If $n$ is odd, then $\Sigma X|_\dX$ naturally identifies with $\Sigma\dX\oplus \Sigma\dX=\Sigma M\oplus \Sigma M$ and an adapted boundary operator is given by 
$$
\D^\dX :=
\begin{pmatrix}
\D & 0\\
0 & -\D
\end{pmatrix} \, .
$$
In both cases, the spectrum of $\D^\dX$ is symmetric about $0$ and $\lambda_1((\D^\dX)^2)=\lambda_1(\D^2)$.

We put $\lambda:=\sqrt{\lambda_1(\D^2)}$ and consider the boundary value problem
\begin{equation}
\begin{cases}
\D^X\phi = 0 \text{ on }X,&\\
\phi|_\dX \in E\big(\D^\dX,(-\infty,\lambda]\big). &
\end{cases}
\label{BVP}
\end{equation}
By Example~7.27 in \cite{BB}, this is an elliptic boundary value problem and the operator
$$
\D^X\colon \{\phi\in H^1(X;\Sigma X) \mid \phi|_\dX \in E\big(\D^\dX,(-\infty,\lambda_1(\D)]\big)\} \to L^2(X;\Sigma X)
$$
is Fredholm.
We write $\ind(\D^X,I)$ for the Fredholm index of the operator 
$$
\D^X\colon \{\phi\in H^1(X;\Sigma X) \mid \phi|_\dX \in E\big(\D^\dX,I\big)\} \to L^2(X;\Sigma X)
$$ 
whenever $I\subset\R$ is an interval giving rise to a Fredholm operator.
By Section~7.2 in \cite{BB}, the adjoint boundary condition of $\phi|_\dX \in E\big(\D^\dX,(-\infty,\lambda]\big)$ is given by $\phi|_\dX \in E\big(\D^\dX,(-\infty,-\lambda)\big)$.
Hence $\ind(\D^X,(-\infty,\lambda])= -\ind(\D^X,(-\infty,-\lambda))$.
By Corollary~8.8 in \cite{BB}, we have that $\ind(\D^X,(-\infty,\lambda])=\ind(\D^X,(-\infty,-\lambda))+\dim\Big(E\big(\D^\dX,[-\lambda,\lambda]\big)\Big)$.
Therefore, 
$$
\ind(\D^X,(-\infty,\lambda]) = \tfrac12 \dim\Big(E\big(\D^\dX,[-\lambda,\lambda]\big)\Big)>0.
$$
Thus, the boundary value problem \eqref{BVP} has a nontrivial solution $\phi$.
By Corollary~7.18 in \cite{BB}, $\phi\in C^\infty(X;\Sigma X)$.
Inserting this solution into the integrated Weitzenböck formula (see e.g.\ Equation~(27) of \cite{BB2}) yields
\begin{align*}
0 
&=
\|\D^X\phi\|_{L^2(X)}^2 \\
&=
\|\nabla^X\phi\|_{L^2(X)}^2 + \int_X \scal^X |\phi|^2 + \tfrac12 \int_M h |\phi|^2 - \big(\D^\dX \phi,\phi\big)_{L^2(M)} \\
&\ge
0 + 0 + \tfrac12 \min_M h\cdot \|\phi\|_{L^2(M)}^2 - \lambda\cdot\|\phi\|_{L^2(M)}^2 
\end{align*}
and therefore
\[
\tfrac12 \min_M h\cdot \|\phi\|_{L^2(M)}^2 \le \lambda\cdot\|\phi\|_{L^2(M)}^2 .
\]
Since $\phi$ is harmonic on $X$ and nontrivial, its restriction to the boundary is also nontrivial by the unique continuation property of Dirac operators.
Thus, 
\[
\min_M h \le 2\lambda.
\]
Now assume $\min_M h = 2\lambda$.
Then we have $\|\nabla^X\phi\|_{L^2(X)}^2=0$, hence $\phi$ is parallel.
The existence of a nontrivial parallel spinor implies that $X$ is Ricci flat.
Moreover, $\phi$ has constant length.
If $h$ were not constant, then $\int_M h |\phi|^2 > \min_M h\cdot \|\phi\|_{L^2(M)}^2$ which would lead to a contradiction.
\end{proof}

%%%%%%%%%%%%%%%%%%%%%%%%%%%%%%%%%%%%%%%%%%%%%%%%%%%%%%%%%%%%%%%%%%%%%%%%%%%%%%%%%%%%%%%%%%%%%%%%%%%%%%%%%%%%%%%%%%%%


\begin{bibdiv}
\begin{biblist}

\bib{B-SpaceForms}{article}{
   author={B\"ar, Christian},
   title={The Dirac operator on space forms of positive curvature},
   journal={J. Math. Soc. Japan},
   volume={48},
   date={1996},
   number={1},
   pages={69--83},
   issn={0025-5645},
   %review={\MR{1361548}},
   doi={10.2969/jmsj/04810069},
}

\bib{B-harmonic}{article}{
   author={B\"ar, Christian},
   title={Metrics with harmonic spinors},
   journal={Geom. Funct. Anal.},
   volume={6},
   date={1996},
   number={6},
   pages={899--942},
   issn={1016-443X},
   %review={\MR{1421872}},
   doi={10.1007/BF02246994},
}

\bib{B-Nodal}{article}{
   author={B\"ar, Christian},
   title={On nodal sets for Dirac and Laplace operators},
   journal={Commun. Math. Phys.},
   volume={188},
   date={1997},
   number={3},
   pages={709--721},
   issn={0010-3616},
   %review={\MR{1473317}},
   doi={10.1007/s002200050184},
}

\bib{BB}{article}{
   author={B\"{a}r, Christian},
   author={Ballmann, Werner},
   title={Boundary value problems for elliptic differential operators of first order},
   %conference={
   %   title={Surveys in differential geometry. Vol. XVII},
   %},
   book={
      series={Surv. Differ. Geom.},
      volume={17},
      publisher={Int. Press, Boston, MA},
   },
   isbn={978-1-57146-237-4},
   date={2012},
   pages={1--78},
   %review={\MR{3076058}},
   doi={10.4310/SDG.2012.v17.n1.a1},
}

\bib{BB2}{article}{
     author={B\"{a}r, Christian},
     author={Ballmann, Werner},
     title={Guide to elliptic boundary value problems for Dirac-type operators},
     %conference={
     %    title={Arbeitstagung Bonn 2013},
     %    },
     book={
         series={Progr. Math.},
         volume={319},
         publisher={Birkh\"{a}user/Springer, Cham},
         },
     date={2016},
     pages={43--80},
     review={\MR{3618047}},
     doi={10.1007/978-3-319-43648-7_3},
}

\bib{BBHW}{article}{
   author={B\"ar, Christian},
   author={Brendle, Simon},
   author={Hanke, Bernhard},
   author={Wang, Yipeng},
   title={Scalar curvature rigidity of warped product metrics},
   journal={SIGMA Symmetry Integrability Geom. Methods Appl.},
   volume={20},
   date={2024},
   pages={Paper No. 035, 26},
   %review={\MR{4733718}},
   doi={10.3842/SIGMA.2024.035},
}

\bib{BGM}{article}{
   author={Brady, Zarathustra},
   author={Guth, Larry},
   author={Manin, Fedor},
   title={A hardness of approximation result in metric geometry},
   journal={Selecta Math. (N.S.)},
   volume={26},
   date={2020},
   number={4},
   pages={Paper No. 54, 20},
   issn={1022-1824},
   %review={\MR{4125986}},
   doi={10.1007/s00029-020-00585-3},
}

\bib{CHS}{arxiv}{
      title={Lipschitz rigidity for scalar curvature}, 
      author={Cecchini, Simone},
      author={Hanke, Bernhard},
      author={Schick, Thomas},
      year={2023},
      url={\url{https://doi.org/10.48550/arXiv.2206.11796}}, 
}

\bib{CHZ}{arxiv}{
      title={Rigidity of spin fill-ins with non-negative scalar curvature}, 
      author={Cecchini, Simone},
      author={Hirsch, Sven},
      author={Zeidler, Rudolf},
      year={2024},
      url={\url{https://doi.org/10.48550/arXiv.2404.17533}}, 
}

\bib{F}{article}{
   author={Friedrich, T.},
   title={Der erste Eigenwert des Dirac-Operators einer kompakten, Riemannschen Mannigfaltigkeit nichtnegativer Skalarkr\"ummung},
   language={German},
   journal={Math. Nachr.},
   volume={97},
   date={1980},
   pages={117--146},
   issn={0025-584X},
   %review={\MR{0600828}},
   doi={10.1002/mana.19800970111},
}

\bib{G}{article}{
   author={Getzler, Ezra},
   title={The odd Chern character in cyclic homology and spectral flow},
   journal={Topology},
   volume={32},
   date={1993},
   number={3},
   pages={489--507},
   issn={0040-9383},
   %review={\MR{1231957}},
   doi={10.1016/0040-9383(93)90002-D},
}

\bib{G4}{article}{
   author={Gromov, Misha},
   title={Four lectures on scalar curvature},
   conference={
      title={Perspectives in scalar curvature. Vol. 1},
   },
   book={
      publisher={World Sci. Publ., Hackensack, NJ},
   },
   isbn={978-981-124-998-3},
   isbn={978-981-124-935-8},
   isbn={978-981-124-936-5},
   date={2023},
   pages={1--514},
   %review={\MR{4577903}},
}

\bib{GL}{article}{
   author={Gromov, Mikhael},
   author={Lawson, H. Blaine},
   title={Spin and scalar curvature in the presence of a fundamental group. I},
   journal={Ann. of Math. (2)},
   volume={111},
   date={1980},
   number={2},
   pages={209--230},
   issn={0003-486X},
   %review={\MR{0569070}},
   doi={10.2307/1971198},
}

\bib{H}{article}{
   author={Hijazi, Oussama},
   title={Premi\`ere valeur propre de l'op\'erateur de Dirac et nombre de Yamabe},
   language={French, with English summary},
   journal={C. R. Acad. Sci. Paris S\'er. I Math.},
   volume={313},
   date={1991},
   number={12},
   pages={865--868},
   issn={0764-4442},
   %review={\MR{1138566}},
}

\bib{HMZ}{article}{
   author={Hijazi, Oussama},
   author={Montiel, Sebasti\'an},
   author={Zhang, Xiao},
   title={Dirac operator on embedded hypersurfaces},
   journal={Math. Res. Lett.},
   volume={8},
   date={2001},
   number={1-2},
   pages={195--208},
   issn={1073-2780},
   %review={\MR{1825270}},
   doi={10.4310/MRL.2001.v8.n2.a8},
}

% \bib{H}{article}{
%    author={Howard, Ralph},
%    title={Alexandrov’s theorem on the second derivatives of convex functions via Rademacher’s theorem on the first derivatives of Lipschitz functions},
%    date={1998},
%    note={lecture notes, available at \url{http://people.math.sc.edu/howard/Notes/alex.pdf}},
% }

\bib{K1}{article}{
   author={Kirchberg, K.-D.},
   title={An estimation for the first eigenvalue of the Dirac operator on closed K\"ahler manifolds of positive scalar curvature},
   journal={Ann. Global Anal. Geom.},
   volume={4},
   date={1986},
   number={3},
   pages={291--325},
   issn={0232-704X},
   %review={\MR{0910548}},
   doi={10.1007/BF00128050},
}

\bib{K2}{article}{
   author={Kirchberg, K.-D.},
   title={The first eigenvalue of the Dirac operator on K\"ahler manifolds},
   journal={J. Geom. Phys.},
   volume={7},
   date={1990},
   number={4},
   pages={449--468 (1991)},
   issn={0393-0440},
   %review={\MR{1131907}},
   doi={10.1016/0393-0440(90)90001-J},
}

\bib{K3}{article}{
   author={Kirchberg, K.-D.},
   title={Eigenvalue estimates for the Dirac operator on K\"ahler-Einstein manifolds of even complex dimension},
   journal={Ann. Global Anal. Geom.},
   volume={38},
   date={2010},
   number={3},
   pages={273--284},
   issn={0232-704X},
   %review={\MR{2721662}},
   doi={10.1007/s10455-010-9212-6},
}

\bib{KSW}{article}{
   author={Kramer, W.},
   author={Semmelmann, U.},
   author={Weingart, G.},
   title={Eigenvalue estimates for the Dirac operator on quaternionic K\"ahler manifolds},
   journal={Math. Z.},
   volume={230},
   date={1999},
   number={4},
   pages={727--751},
   issn={0025-5874},
   %review={\MR{1686563}},
   doi={10.1007/PL00004715},
}

\bib{KMS}{article}{
   author={Kupferman, Raz},
   author={Maor, Cy},
   author={Shachar, Asaf},
   title={Reshetnyak rigidity for Riemannian manifolds},
   journal={Arch. Ration. Mech. Anal.},
   volume={231},
   date={2019},
   number={1},
   pages={367--408},
   issn={0003-9527},
   %review={\MR{3894554}},
   doi={10.1007/s00205-018-1282-9},
}

\bib{LT}{arxiv}{
   title={Rigidity of Lipschitz map using harmonic map heat flow}, 
   author={Lee, Man-Chun},
   author={Tam, Luen-Fai},
   year={2022},
   url={\url{https://doi.org/10.48550/arXiv.2207.11017}}, 
}

\bib{LSW}{article}{
   author={Li, Yihan},
   author={Su, Guangxiang},
   author={Wang, Xiangsheng},
   title={Spectral flow, Llarull's rigidity theorem in odd dimensions and its generalization},
   journal={Sci. China Math.},
   volume={67},
   date={2024},
   number={5},
   pages={1103--1114},
   issn={1674-7283},
   %review={\MR{4739559}},
   doi={10.1007/s11425-023-2138-5},
}

\bib{Li}{arxiv}{
   title={Scalar curvature on compact symmetric spaces}, 
   author={Listing, Mario},
   year={2010},
   url={\url{https://doi.org/10.48550/arXiv.1007.1832}}, 
}

\bib{Ll}{article}{
   author={Llarull, Marcelo},
   title={Sharp estimates and the Dirac operator},
   journal={Math. Ann.},
   volume={310},
   date={1998},
   number={1},
   pages={55--71},
   issn={0025-5831},
   %review={\MR{1600027}},
   doi={10.1007/s002080050136},
}

% \bib{M}{article}{
%    author={Miao, Pengzi},
%    title={Positive mass theorem on manifolds admitting corners along a
%    hypersurface},
%    journal={Adv. Theor. Math. Phys.},
%    volume={6},
%    date={2002},
%    number={6},
%    pages={1163--1182 (2003)},
%    issn={1095-0761},
%    %review={\MR{1982695}},
%    doi={10.4310/ATMP.2002.v6.n6.a4},
% }

\bib{MO}{article}{
   author={Moroianu, Andrei},
   author={Ornea, Liviu},
   title={Eigenvalue estimates for the Dirac operator and harmonic 1-forms of constant length},
   language={English, with English and French summaries},
   journal={C. R. Math. Acad. Sci. Paris},
   volume={338},
   date={2004},
   number={7},
   pages={561--564},
   issn={1631-073X},
   %review={\MR{2057030}},
   doi={10.1016/j.crma.2004.01.030},
}

\bib{ST}{article}{
   author={Shi, Yuguang},
   author={Tam, Luen-Fai},
   title={Positive mass theorem and the boundary behaviors of compact manifolds with nonnegative scalar curvature},
   journal={J. Diff. Geom.},
   volume={62},
   date={2002},
   number={1},
   pages={79--125},
   issn={0022-040X},
   %review={\MR{1987378}},
}

% \bib{R}{article}{
%    author={Rader, Trout},
%    title={Nice demand functions},
%    journal={Econometrica},
%    volume={41},
%    date={1973},
%    pages={913--935},
%    issn={0012-9682},
%    %review={\MR{0441276}},
%    doi={10.2307/1913814},
% }

\end{biblist}
\end{bibdiv}
\end{document}